\documentclass{elsarticle}

\usepackage{amsmath}
\usepackage{amssymb}
\usepackage{amsthm}
\usepackage{tikz}
\usepackage{stmaryrd}
\usepackage{enumitem}
\usepackage{pgfplots}
\usepackage{pgfplotstable}
\usepackage{subcaption}

\newtheorem{problem}{Problem}
\newtheorem{theorem}{Theorem}

\theoremstyle{remark}
\newtheorem{remark}{Remark}

\newcommand{\dx}[0]{\,\mathrm{d}\boldsymbol{x}}
\newcommand{\Ch}[0]{\mathcal{C}_h}
\newcommand{\Eh}[0]{\mathcal{E}_h}
\newcommand{\Div}[0]{\mathrm{div}\,}
\newcommand{\Eop}[1]{\mathcal{A}(#1)}
\newcommand{\Eoph}[1]{\mathcal{A}_h(#1)}

\newcommand{\vertiii}[1]{{\left\vert\kern-0.25ex\left\vert\kern-0.25ex\left\vert #1 
    \right\vert\kern-0.25ex\right\vert\kern-0.25ex\right\vert}}
\newcommand{\enorm}[1]{\vertiii{#1}}
\newcommand{\edual}[1]{\vertiii{#1}_*}
\newcommand{\edualK}[1]{\vertiii{#1}_{*,K}}

\newcommand{\Bf}[0]{\mathcal{B}}
\newcommand{\Lf}[0]{\mathcal{L}}

\begin{document}

\title{An adaptive finite element method for the inequality-constrained Reynolds equation}
\begin{abstract}
We present a stabilized finite element method for the numerical solution of cavitation in lubrication,  modeled as an inequality-constrained Reynolds equation. The cavitation model is written as  a variable coefficient saddle-point problem and approximated by a residual-based stabilized method. Based on our recent results on the classical obstacle problem, we present optimal a priori estimates and derive novel a posteriori error estimators.

The method is implemented as a Nitsche-type finite element technique and shown in numerical computations to be superior to the usually applied penalty methods. 
\end{abstract}
\author[aalto]{Tom~Gustafsson}
\author[texasam]{Kumbakonam R.~Rajagopal}
\author[aalto]{Rolf~Stenberg}
\author[ist]{Juha~Videman}

\address[aalto]{Aalto University, Department of Mathematics and Systems Analysis, P.O. Box 11100 FI-00076 Aalto, Finland}
\address[texasam]{Department of Mechanical Engineering, Texas A\&M University, 3123 College Station, TX 77843-3123, USA}
\address[ist]{CAMGSD and Mathematics Department, Instituto Superior T\'ecnico, Universidade de Lisboa, Av. Rovisco Pais 1, 1049-001 Lisboa, Portugal}

\maketitle

\section{Introduction}


Friction is an ever present phenomenon between solid surfaces that are
in contact and move with respect to each other. In order to reduce the
wear and tear of such sliding surfaces
one uses a lubricant.

The lubricant film thickness is often small in comparison to the
characteristic dimensions of the lubricated object. Assuming that the
flow in such thin films can be modeled by the Navier--Stokes fluid,
Reynolds obtained an approximation to the equations of
motion~\cite{reynolds1886theory}. 
This approximation, known as the  Reynolds equation, has
proved to be invaluable, however  the equations have unfortunately also
been misused and misapplied to various problems wherein the basic
assumptions that Reynolds made are not met. In some cases,
generalizations have been made
that are incorrect (see the discussion in Rajagopal
and Szeri~\cite{RS03}).

 Although liquids can sustain negative pressures to some extent, 
 it is of particular interest to
model the  cavitation phenomenon which is known to occur in
tribological environments under strong sub-atmospheric pressures,
cf.~Braun--Hannon~\citep{braun2010cavitation} for a fairly recent
review on the subject from a modeling perspective. The described
physical conditions result in a rupture of the lubricant film causing
the formation of cavities that are filled by a mixture of gas and
vapor, cf. the experimental results and pictures, {\sl e.g.}, in Xing et al.~\citep{xing2009three}.

A direct application of the Reynolds equation in the
full 360$^\circ$ journal bearing geometry leads to a pressure field with
equal positive and negative pressure spikes in the converging and
diverging regions,
respectively. 
As this obviously cannot hold true
for large negative pressures due to cavitation, different
attempts have been made to incorporate cavitation  into the Reynolds
equation.

The most rudimentary approach for taking into account the resulting
film rupture---often accredited to
G\"umbel~\citep{gumbel1914problem}---is to simply compute the pressure
field $p$ from the standard Reynolds equation and
truncate it wherever it goes below a predetermined constant cavitation pressure value. Another possibility, put forward by
Swift~\cite{swift1932stability} and
Stieber~\cite{stieber1933schwimmlager}, is to impose the following
cavitation zone formation conditions at the cavitation boundary:
\begin{equation}
  \label{eq:swiftstieber}
  \frac{\partial p}{\partial x} = \frac{\partial p}{\partial y} = 0, \qquad p = p_c.
\end{equation}
From a mathematical perspective these conditions transform the model into a
free boundary problem that can be formulated as a variational
inequality~\cite{rodrigues1987obstacle}.

The Swift--Stieber or the variational inequality model has been
criticized for not properly modeling the starvation phenomena
where the cavitation region extends to the converging side
of the flow region, cf.
Bayada--V\'azquez~\citep{bayada2007survey}. This has led to the development of
alternative Reynolds-type cavitation models, see, {\sl e.g.}, ~Elrod--Adams~\citep{elrod1974computer},
Vijayaraghavan--Keith~\citep{vijayaraghavan1989development}, Bayada et
al.~\citep{bayada1990variational,
bayada2001finite, bayada2005average, bayada1983freebdry}, Almqvist et
al.~\citep{almqvist2014new}, Garcia et al.~\citep{garcia2017piezo},
Mistry et al.~\citep{biswas1997new}, for
different models and their applications. It has also been recently shown, within the context of fluids with pressure dependent viscosity, by Lanzend\"orfer et al.~\cite{LMR17}  that a solution based on using a cut-off value for the pressure to determine the region of cavitation is sensitive to the value of the cut-off.
Nevertheless, it is obvious that any cavitation model  derived from continuum mechanics involves an initially
unknown region of cavitation and therefore it is mandatory to consider numerical methods that
automatically capture the location of the cavitation region and adapt to the resulting free boundaries.

As a step towards better numerical modeling of cavitation, we study
the numerical solution of the Reynolds model with Swift--Stieber
cavitation conditions using adaptive finite element
techniques. Similarly to our recent work on the obstacle
problem~\citep{GSV17}, the positivity constraint of the pressure field
is imposed using the Lagrange multiplier technique instead of the more
traditional penalty method since adaptive penalty methods easily lead
to over-refinement in the cavitation area~\citep{juntunen2009nitsche,
sorsimo2012finite}.  The resulting variational inequality has a
saddle-point structure and must be solved either by using a carefully
designed pair of mixed finite element spaces or by a stabilized method in
the spirit of Hughes--Franca~\citep{hughes1987new}. In this work, we
choose the latter approach since  the discontinuous Lagrange multiplier can be eliminated
elementwise from  the discrete linearized problem, cf. Section~\ref{implementation}.  The resulting formulation involves only positive
definite linear systems that are fast to assemble~\citep{GSV17}.

Cavitation is a very daunting
problem which is yet to be thoroughly  addressed.
In fact,
if one were to carry out a meaningful study of the problem, one should  adopt a fully thermodynamic
approach as there can be dissipation in the viscous lubricant 
leading to an increase of the lubricant temperature. In certain
applications, this increase can be significant thereby
leading to a change in the viscosity.
Complicating matters further is the fact that the cavity is not
necessarily a vacuum but contains the vapor of the lubricant and
possibly trapped air, in other words a material for which one needs a thermodynamic equation of
state. Temperature can have a major
role in the initiation of cavitation and also in determining the vapor
pressure within the cavity. Moreover, one may need to take into account
the surface tension at the boundary of the cavity and  solve 
the interaction problem between the fluid and the cavity that is constantly
changing with space and time.  Furthermore, the state of the fluid
prior to the initiation of cavitation can be such that one might
have to take compressibility effects into consideration. Of course,
one cannot account for all these issues simultaneously as the
problem would become intractable. A sensible approach is  to
consider first the simplest approximation that yet captures the
quintessential features of the problem, and  then add more
of the physically relevant issues. This is precisely the approach that
we adopt in this work. 

For previous studies of adaptive finite element techniques applied to
cavitation modeling, we refer to Wu--Oden~\citep{wu1987note},
Nilsson--Hansbo~\citep{nilsson2007adaptive} and Sorsimo et
al.~\citep{sorsimo2012finite}. In all these works, 
the classical penalty method was used to enforce the constraint $p \geq
p_c$. The error estimator in the early study of
Wu--Oden~\citep{wu1987note} was based on the simplistic idea of
minimizing the large gradients present in the solution. The more recent works
of Nilsson--Hansbo~\citep{nilsson2007adaptive} and Sorsimo et
al.~\citep{sorsimo2012finite}  were built upon the a posteriori estimates
derived by Johnson~\citep{johnson1992adaptive} for the obstacle
problem in the penalty formulation.

Some shortcomings of the penalty method approach, when applied to the obstacle
problem of an elastic membrane, are discussed
in~\citep{gustafsson2017obstacle}.  They include:
\begin{itemize}
    \item Nonconformity of the  method. This means that the
          convergence rate is optimal 
          for linear elements only.
    \item  Condition numbers of the resulting linear systems increase faster
          than for standard finite elements of the same polynomial order.
    \item The method can lead to over-refinement if the penalty parameter
          is not chosen appropriately.
\end{itemize}
As shown in our earlier works~\citep{GSV17,
gustafsson2017obstacle}, none of these issues arise from the
stabilized finite element formulation of  variational inequalities. Thus, the
main goal of this work is to investigate the suitability of stabilized finite
element methods for the approximation of the Reynolds' cavitation model and, at
the same time, establish novel a posteriori error estimators for variable
coefficient elliptic variational inequalities. To our knowledge, this is the
first time when suitably modified residual based a posteriori error estimators
and stabilizing terms have been presented for variable coefficient saddle-point
problems, see~\cite{BV2000, DW2000} for similar a posteriori error analyses of
elliptic equations. 


\begin{figure}[h]
    \centering
    \begin{tikzpicture}
        \draw (0,0) -- (6,0) -- (6,4) -- (0,4) -- (0,0);
        \draw (6,2) node[anchor=west] {$\Gamma_2$};
        \draw (0,2) node[anchor=east] {$\Gamma_4$};
        \draw (3,4) node[anchor=south] {$\Gamma_3$};
        \draw (3,0) node[anchor=north] {$\Gamma_1$};
        \draw[-latex] (0.1,0.1) -- (1,0.1) node[above] {$\theta$};
        \draw[-latex] (0.1,0.1) -- (0.1,1) node[right] {$y$};
    \end{tikzpicture}
    \caption{Let $p$ denote the pressure field. Typical boundary conditions for
        the pressure field in the full periodic 360$^\circ$ journal bearing geometry are
        $p|_{\Gamma_1}=p|_{\Gamma_3}=0$ (constant atmospheric pressure) and
        $p(\theta,y)=p(\theta+2\pi,y)$, $\theta \in \Gamma_4$ (periodic
        boundary condition). Moreover, for the periodic boundaries we have
        $\frac{\partial p}{\partial \theta}(\theta,y) = \frac{\partial
        p}{\partial \theta}(\theta+2\pi,y)$, $\theta \in \Gamma_4$.  Sometimes
        (for example in locomotive applications) the journal is in contact with
        the bearing surface only on a part of the journal boundary. Then the
        width of the domain is less than $2\pi$ and the boundary conditions are
        simply
        $p|_{\Gamma_i}=0$, $i \in \{1,2,3,4\}$.}
    \label{fig:journalbearingbc}
\end{figure}
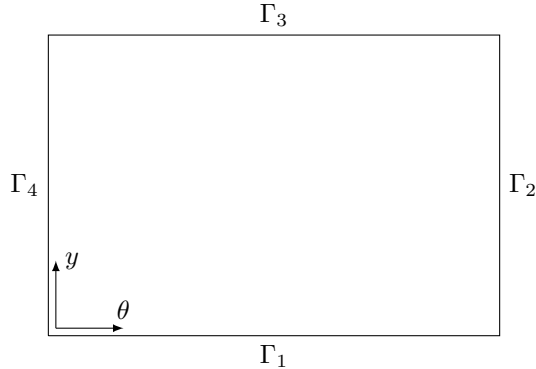

\section{Reynolds equation}

Derivation of the Reynolds equation from the Navier--Stokes
equations is standard and can be found, {\sl e.g.}, in~\citep{szeri2005fluid,
hamrock2004fundamentals}. The assumptions  are
the following:
\begin{itemize}
  \item The lubricant is homogeneous, incompressible and Newtonian, i.e.~the Navier--Stokes equations are valid.
  \item The flow takes place between two almost parallel surfaces with
        the distance between the surfaces given by a function $d :
        \mathbb{R}^2 \rightarrow \{x \in \mathbb{R} : x > 0 \}$.
  \item The fluid film thickness $d$ is small with respect to the characteristic lengths along other directions.
  \item The curvature of the surfaces is negligible.
  \item The flow is slow enough or the viscosity $\mu$ high enough so that the lubricant inertia can be ignored.
\end{itemize}
In this work, we consider steady-state flow regime and assume that the
viscosity $\mu$ remains constant. Note, however, that our results are valid
also for the equation arising from a fixed-point linearization of the Reynolds
equation with a moderate pressure-dependency of $\mu$ as shown
in~\citep{gustafsson2015nonlinear}. Compressibility, curvature and inertia
effects can be taken into account by a suitable modification of the Reynolds
equation, for the latter two we refer to \cite{NV07}.

Let $\Omega \subset \mathbb{R}^2$~be the computational domain representing 
both the lubricated and the cavitated region. Assume that the relative velocity between the surfaces
has magnitude $U$ and the direction is (locally) parallel to the $x$-axis. 
The Reynolds equation for the unknown pressure field $p : \Omega \rightarrow \mathbb{R}$
reads as
\begin{equation}
  \label{eq:reynolds}
  -\Div \left( \frac{d^3}{\mu} \nabla p\right) = -6 U \frac{\partial d}{\partial x}.
\end{equation}
Equation \eqref{eq:reynolds} is complemented  with appropriate boundary conditions for  $p$.

\begin{remark}
In the journal bearing geometry,
we have $x = R\theta$, $U = R \omega$ and $d(\theta) = c_1 (1 + \varepsilon
\cos (\theta-\varphi))$ where $R$ is the radius of the bearing, $\omega$ is the
angular velocity and $(c_1,\varepsilon)$ are geometrical parameters;
$\varepsilon \in (0,1)$ is the eccentricity of the bearing, $c_1$ is the
minimum film thickness and $\varphi$ determines the location of the minimum film
thickness (due to external loading on the journal).  The Reynolds equation becomes
\begin{equation}
  \label{eq:reynoldsjournal}
  -\frac{\partial}{\partial \theta}\left( d^3 \frac{\partial p}{\partial \theta}\right) - R^2 \frac{\partial}{\partial y}\left( d^3 \frac{\partial p}{\partial y}\right) = -6 \mu \omega R^2 \frac{\partial d}{\partial \theta} \quad \text{in $\Omega$},
\end{equation}
where
\[
\Omega =\{ (\theta,y) \, :\, 0\leq \theta \leq \psi, \quad 0\leq y \leq L \}\, .
\]
Here, $L$ is the width of the finite bearing and $0 < \psi \leq 2 \pi$ defines
the extent of the bearing surface. The boundary conditions  are shown in
Figure~\ref{fig:journalbearingbc} and are different for $\psi = 2\pi$ and $\psi
< 2\pi$.
\end{remark}

\begin{remark}
  The Reynolds equation \eqref{eq:reynoldsjournal} can be normalized
 defining the non-dimensionalized variables
  \begin{equation}
    \overline{d} = \frac{d}{c_1}, \quad \overline{y} = \frac{y}{L}, \quad \overline{p} = \frac{p}{\mu \omega}\left( \frac{c_1}{R}\right)^2.
  \end{equation}
  This leads to
  \begin{equation}
    -\frac{\partial}{\partial \theta} \left( \overline{d}^3 \frac{\partial \overline{p}}{\partial \theta} \right) - \left( \frac{R}{L} \right)^2 \frac{\partial}{\partial \overline{y}}\left( \overline{d}^3 \frac{\partial \overline{p}}{\partial \overline{y}} \right) = - 6 \frac{\partial \overline{d}}{\partial \theta}.
  \end{equation}
\end{remark}


The Swift--Stieber cavitation model corresponds to imposing the constraint $p
\geq p_c$ where $p_c$ is the cavitation pressure.  The momentum balance
\eqref{eq:reynolds} holds true only in the lubricated region where $p$ is
strictly greater than $p_c$. In the cavitation region  $p=p_c$. Hence, the
governing equations read as
\begin{alignat}{2}
    -\Eop{p} &\geq f, \quad && \text{in $\Omega$,} \label{cavmodel1} \\ 
    p-p_c &\geq 0, \quad  && \text{in $\Omega$,} \label{cavmodel2}  \\
    (\Eop{p} +f)(p-p_c) &= 0,   \quad && \text{in $\Omega$,}  \label{cavmodel3}  \,  \\
\end{alignat}
where we have set $f=-6\mu U \frac{\partial d}{\partial x}$ and defined
\begin{equation}
  \Eop{p} = \Div(d^3 \nabla p).
\end{equation}
The boundary conditions are
\begin{align}
    p|_{\Gamma_1 \cup \Gamma_3}&=0,\label{cavmodel4}  \\
    p|_{\Gamma_{2}}&=p|_{\Gamma_{4}}, \label{cavmodel5}  \\ 
    (\nabla p \cdot \boldsymbol{n})|_{\Gamma_{2}}&=(\nabla p \cdot \boldsymbol{n})|_{\Gamma_{4}}, \label{cavmodel6} 
\end{align}
for the full 360$^\circ$ journal bearing and
\begin{align}
    p|_{\partial \Omega}&=0,\label{cavmodel7}
\end{align}
for smaller contact angles, see the explanations in the caption of
Figure~\ref{fig:journalbearingbc}.  Note that the cavitation zone formation
conditions \eqref{eq:swiftstieber} are implicitly satisfied inside and on the
boundary of the cavitation region.

In the following chapters we formulate the variational inequality using the
simpler boundary condition \eqref{cavmodel7} although everything holds true
also in the case of a periodic boundary condition.

\section{Variational inequality and its discretization}

Aiming at solving the  problem \eqref{cavmodel1}--\eqref{cavmodel6} by the finite element method, we first write it
in a variational form. A method best suited to adaptive solution strategies includes an additional unknown  $\lambda :\Omega \rightarrow \{ x \in \mathbb{R} : x \geq 0 \}$, a Lagrange multiplier function positive
in the cavitation region and zero elsewhere. The 
equations \eqref{cavmodel1}--\eqref{cavmodel3}  are then transformed to
\begin{align}
    \label{eq:lagstrong1}-\Eop{p} - \lambda &= f, \\
    \label{eq:lagstrong2}p-p_c &\geq 0, \\
    \label{eq:lagstrong3}(p-p_c)\lambda &= 0. 
\end{align}

Let us introduce the variational spaces 
\begin{align}
    Q &= \{ q \in H^1(\Omega) : q|_{\partial \Omega}=0 \}, \\
    \varLambda &= \{ \mu \in H^{-1}(\Omega) : \langle q, \mu \rangle \geq 0 ~\forall q \in Q,~q \geq 0 ~\text{a.e.~in $\Omega$} \},
\end{align}
equipped with the energy norm $\enorm{\cdot}$ and  its  dual norm $\edual{\cdot}$ defined by
\begin{equation}
  \enorm{r}^2 = \int_\Omega d^3 (\nabla r)^2 \dx \quad \text{and} \quad \edual{\xi} = \sup_{q \in Q} \frac{\langle q, \xi \rangle}{\enorm{q}}\, ,
\end{equation}
where $\langle \cdot, \cdot \rangle : Q \times H^{-1}(\Omega)
\rightarrow \mathbb{R}$ denotes the duality pairing.
Defining the bilinear and linear forms through
\begin{align*}
  \Bf(r,\xi; q, \mu) &= \int_\Omega d^3 \nabla r \cdot \nabla q \dx - \langle q, \xi \rangle - \langle r, \mu \rangle, \\
  \Lf(q, \mu) &= \int_\Omega f q \dx - \langle p_c, \mu \rangle \, ,
\end{align*}
we arrive at the following variational inequality formulation of  problem \eqref{eq:lagstrong1}--\eqref{eq:lagstrong3}, \eqref{cavmodel4}--\eqref{cavmodel6}.
\begin{problem}[Continuous variational problem]
    Find $(p,\lambda) \in Q \times \varLambda$ satisfying
    \begin{equation}
      \label{eq:lagweak} \Bf(p,\lambda; q, \mu-\lambda) \leq \Lf(q, \mu-\lambda)
    \end{equation}
    for every $(q,\mu) \in Q \times \varLambda$.
    \label{prob:cont}
\end{problem}

\begin{remark}
  Formulation
    \eqref{eq:lagweak} can also be derived by applying
    the method of Lagrange multipliers to the  constrained
    energy minimization problem
    \begin{equation}
    \inf_{q \in Q,~q\geq p_c} \frac12\enorm{q}^2 - \int_{\Omega} fq \dx.
    \end{equation}
\end{remark}

 The following stability estimate guarantees the well-posedness of Problem~\ref{prob:cont}.  The stability result of Theorem \ref{contstab} is
also crucial for the a posteriori error analysis presented in the
next section.

\begin{theorem}[Continuous stability]
  For all $(r,\xi) \in Q \times H^{-1}(\Omega)$ there exists $q \in Q$ such that with some constants $C_1,C_2>0$ it holds
  \begin{equation}
    \Bf(r,\xi; q, -\xi) \geq C_1 \left(\enorm{r} + \edual{\xi}\right)^2 \quad \text{and} \quad \enorm{q} \leq C_2 \left( \enorm{r} + \edual{\xi} \right)\, .
  \end{equation}

\label{contstab}
\end{theorem}

\begin{proof}
  Choose $q = r - r_\xi$ where $r_\xi \in Q$. Then by linearity and Cauchy--Schwarz inequality
  \begin{equation}
    \label{proof:contstab1}
    \Bf(r,\xi;r-r_\xi,-\xi) \geq \enorm{r}^2 - \enorm{r}\,\enorm{r_\xi} + \langle r_\xi, \xi \rangle.
  \end{equation}
  We now fix $r_\xi \in Q$ as the solution of the auxiliary problem
  \begin{equation}
    \int_\Omega d^3 \nabla r_\xi \cdot \nabla q \dx = \langle q, \xi \rangle, \quad \forall q \in Q.
  \end{equation}
  For such $r_\xi$, a standard reasoning shows that $\langle r_\xi, \xi \rangle = \enorm{r_\xi} = \edual{\xi}$, cf.~\citep{GSV17}. The upper bound for  $\enorm{r} + \edual{\xi}$ now follows from Young's inequality. The lower
bound is a  direct consequence of  the triangle inequality.
\end{proof}

We will approximate the (saddle-point)  Problem~\ref{prob:cont}  by  a
stabilized finite element method, instead of a mixed method, to avoid the Babu\v{s}ka-Brezzi inf-sup condition. Let $\Ch$~denote the discretization
of  $\Omega$ into triangular elements $K \in \Ch$ with $h_K$
referring to the local mesh parameter and $h = \max_{K \in \Ch}
h_K$. The discrete finite element spaces $Q_h \subset Q$ and
$\varLambda_h \subset L^2(\Omega)$ are defined as
\begin{align}
    Q_h &= \{ q_h \in Q : q_h|_K \in P_k(K)~\forall K \in \Ch \},\\
  \varLambda_h &= \{ \mu_h \in L^{2}(\Omega) : \mu_h|_K \in P_l(K)~\forall K \in \Ch \},
\end{align}
where $k\geq1$ and $l\geq0$ are the polynomial degrees.
We further denote by $\varLambda_h^+$ the following subset of  $\varLambda_h$
\begin{equation}
  \varLambda_h^+ = \{ \mu_h \in \varLambda_h : \mu_h \geq 0 \text{ a.e.~in $\Omega$} \}.
\end{equation}

Let $\alpha > 0$ be a stabilization parameter.
We introduce the following discrete bilinear and linear forms
\begin{align*}
  \Bf_h(r,\xi; q, \mu) &= \Bf(r,\xi; q, \mu) - \alpha \sum_{K \in \Ch} \frac{h_K^2}{d_K^3} \int_K(-\Eop{r} - \xi)(-\Eop{q} - \mu) \dx, \\
  \Lf_h(q, \mu) &= \Lf(q, \mu) - \alpha \sum_{K \in \Ch} \frac{h_K^2}{d_K^3} \int_K f (-\Eop{q} - \mu) \dx,
\end{align*}
where $d_K$ represents the mean value of $d$ inside the element $K$.
Note that, in contrast to the obstacle problem with a constant material
parameter, here a proper scaling by $d_K$ is mandatory as the film thickness function
$d^3$ may vary several orders of magnitude inside the domain in practical
lubrication problems.


\begin{problem}[Stabilized finite element method]
    Find $(p_h,\lambda_h)
    \in Q_h \times \varLambda_h^+$~satisfying
    \begin{equation}
      \label{eq:discweak} \Bf_h(p_h,\lambda_h; q_h, \mu_h-\lambda_h) \leq \Lf_h(q_h, \mu_h-\lambda_h)
    \end{equation}
    for every $(q_h,\mu_h) \in Q_h \times \varLambda_h^+$.
    \label{prob:stabfem}
\end{problem}

To simplify the
estimates, we will 
assume that $d \in C^1(\Omega)$ and that for a given mesh $\Ch$ there
exists global positive constants $\gamma_1$ and $\gamma_2$ so that
\begin{equation}
  \gamma_1 \min_{\boldsymbol{x} \in K} d(\boldsymbol{x}) \leq d(\boldsymbol{y}) \leq \gamma_2 \max_{\boldsymbol{x} \in K} d(\boldsymbol{x})
\end{equation}
for any element $K \in \Ch$ and at every point $\boldsymbol{y} \in K$. Generalization to a non-smooth $d$ would be straightforward following the reasoning presented in \cite{BV2000}, see also \cite{Verf2000}.

\begin{theorem}[Discrete stability]
  \label{thm:discstab}
  Suppose that $0<\alpha<C_I$ where $C_I$ is the constant of the
  inverse inequality
  \begin{equation}
    C_I \sum_{K \in \Ch} \frac{h_K^2}{d_K^3} \|\Eop{q_h}\|_{0,K}^2 \leq \enorm{q_h}^2 \quad \forall q_h \in Q_h. 
    \label{inverseineq}
  \end{equation}
Then for all $(r_h,\xi_h) \in Q_h \times \varLambda_h$ there exists $q_h \in Q_h$ such that for some $C_1,C_2>0$ it holds
  \begin{equation}
    \Bf_h(r_h,\xi_h; q_h, -\xi_h) \geq C_1 \left(\enorm{r_h} + \edual{\xi_h}\right)^2,
  \end{equation}
  and
  \begin{equation}
    \enorm{q_h} \leq C_2 \left( \enorm{r_h} + \edual{\xi_h} \right).
  \end{equation}
\end{theorem}

\begin{proof}
    Using the inverse inequality~\eqref{inverseineq}, we get
  \begin{align*}
    \Bf_h(r_h,\xi_h; r_h, -\xi_h) &= \enorm{r_h}^2 - \alpha \sum_{K \in \Ch} \frac{h_K^2}{d_K^3} \| \Eop{r_h} \|_{0,K}^2 + \alpha \sum_{K \in \Ch} \frac{h_K^2}{d_K^3} \|\xi_h\|_{0,K}^2 \\
&\geq \left(1-\frac{\alpha}{C_I}\right) \enorm{r_h}^2 + \alpha \sum_{K \in \Ch} \frac{h_K^2}{d_K^3} \|\xi_h\|_{0,K}^2.
  \end{align*}
  Following the general steps presented in \citep[Theorem 4.1]{GSV17}, we can
  construct $\widetilde{q_h} \in Q_h$ which satisfies
  \begin{align*}
    \Bf_h(r_h,\xi_h; \widetilde{q_h}, 0) \geq C_3 \edual{\xi_h}^2 - C_4\left(\enorm{r_h}^2 + \sum_{K \in \Ch} \frac{h_K^2}{d_K^3} \|\xi_h\|_{0,K}^2 \right)
  \end{align*}
  with some positive constants $C_3$ and $C_4$. Hence by bilinearity
  \begin{align*}
    \Bf_h(r_h,\xi_h; r_h + \delta \widetilde{q_h}, -\xi_h) &= \Bf_h(r_h,\xi_h; r_h, -\xi_h) + \delta \Bf_h(r_h,\xi_h; \widetilde{q_h}, 0) \\
                                                           &\geq  \left( 1 - \frac{\alpha}{C_I} - \delta C_4 \right) \enorm{r_h}^2 + \delta C_3 \edual{\xi_h}^2 \\
    &\phantom{=}\quad + \left(\alpha - \delta C_4\right) \sum_{K \in \Ch} \frac{h_K^2}{d_K^3} \|\xi_h\|_{0,K}^2,
  \end{align*}
  where for any given $\alpha \in (0,C_I)$ the positive parameter
  $\delta$ can be chosen in such a way that all terms on the right hand side
  remain positive.
\end{proof}

\begin{theorem}
    Let $f_h \in Q_h$ be the $L^2$-projection of $f$ and suppose that
    $0<\alpha<C_I$. There exists $C>0$ such that the following a priori
    estimate holds
    \begin{align}
        &\enorm{p-p_h} + \edual{\lambda-\lambda_h} \\
        &\leq C \left( \inf_{q_h \in Q_h} \enorm{p-q_h} + \inf_{\mu_h \in \varLambda_h^+}\left(\edual{\lambda-\mu_h}+\sqrt{\int_\Omega (p-p_c)\,\mu_h \dx} \right)\right.\nonumber\\
                                                 &\left. \phantom{=}\quad+  \sqrt{\sum_{K \in \Ch} \tfrac{h_K^2}{d_K^3}\|f-f_h\|_{0,K}^2 } \right).\nonumber
    \end{align}
\end{theorem}


This is essentially a best approximation result analogous to the well
known C\'{e}a's lemma which holds for coercive variational
problems. The proof is based on the stability result of
Theorem~\ref{thm:discstab} and can be found in~\cite{GSV17}. The first two
terms on the right hand side can be bounded from above using interpolation
estimates. The third term is zero if the finite element mesh follows exactly
the boundary of the cavitation region. In general this cannot be guaranteed a
priori and adaptive methods are required to properly resolve the cavitation
boundary.

\section{Adaptive refinement}

\label{sec:adaptive}

Let $d_E$ stand for $d$'s mean value  along the  edge $E$. The following a posteriori error bounds
demonstrate the suitability of our stabilized finite element method to adaptive solution
strategies.

\begin{theorem}
    \label{thm:aposteriori}
    The following a posteriori estimate holds
    \begin{align}
        &\enorm{p-p_h}+\edual{\lambda-\lambda_h} \\
        &\leq C\left(\sqrt{\sum_{K \in \Ch} \tfrac{h_K^2}{d_K^3} \|\Eop{p_h}+\lambda_h+f\|_{0,K}^2 } + \sqrt{\sum_{E \in \Eh} \tfrac{h_E}{d_E^3} \| d^3 \llbracket \nabla p_h \cdot \boldsymbol{n} \rrbracket\|_{0,E}^2} \right. \nonumber \\
        &\phantom{=}\quad \left. +\enorm{(p_c-p_h)_+} +\sqrt{\int_\Omega (p_h-g)_+\lambda_h \dx}\right) \, , 
        \nonumber 
    \end{align}
    where the constant $C>0$ is independent of $h$ and $d$. 
\end{theorem}
\begin{proof} Since the proof goes as in~\cite{GSV17}, we will only outline the main steps and refer to~\cite{GSV17} for further details.

In view of  Theorem \ref{contstab}, there exists $q\in Q$ such that
\[
\enorm{q} \leq C_2 \left( \enorm{p-p_h} + \edual{\lambda-\lambda_h} \right)
\]
and 
    \begin{align*}
        & (\enorm{p-p_h}+\edual{\lambda-\lambda_h})^2 \\
    & \leq C_1\,   \Bf(p-p_h,\lambda-\lambda_h; q, \lambda_h-\lambda) \\
    & \leq
    C_1\, \big(  \Bf(p,\lambda; q, \lambda_h-\lambda) -\Bf(p_h,\lambda_h; q, \lambda_h-\lambda) -\Bf_h(p_h,\lambda_h;-\widetilde{q},0)+\Lf_h(-\widetilde{q},0)\big) \\
    &\leq C_1\Big( \Lf(q-\widetilde{q},\lambda_h-\lambda) -\Bf(p_h,\lambda_h; q-\widetilde{q}, \lambda_h-\lambda) \\
     &\phantom{=}\qquad  -\alpha 
    \sum_{K \in \Ch} \tfrac{h_K^2}{d_K^3}(\Eop{p_h}+\lambda_h+f,\Eop{\widetilde{q}})_{0,K}\Big) \\
    & = C_1\Big( \sum_{K \in \Ch} (\Eop{p_h}+\lambda_h+f,q-\widetilde{q})_{0,K} -\sum_{E \in \Eh} (d^3 \llbracket \nabla p_h \cdot \boldsymbol{n} \rrbracket,q-\widetilde{q})_{0,E}\\
    &\phantom{=}\qquad  +\langle p_h-p_c,\lambda_h-\lambda\rangle  -\alpha 
    \sum_{K \in \Ch} \tfrac{h_K^2}{d_K^3}(\Eop{p_h}+\lambda_h+f,\Eop{\widetilde{q}})_{0,K}\Big) \, ,
    \end{align*}
where $\widetilde{q}\in Q_h$ is the Cl\'ement interpolant of $q$. Note that
\[
0\leq -\Bf_h(p_h,\lambda_h;-\widetilde{q},0)+\Lf_h(-\widetilde{q},0)
\]
and that for some positive constants $C^\prime$ and $C^{\prime\prime}$ it holds
\[
\Big(  \sum_{K \in \Ch}  h_K^{-2} d_K^3 \| q-\widetilde{q} \|_{0,K}^2\Big)^{1/2} + \Big(  \sum_{E \in \Eh} h_E^{-1}d_E^3 \| q-\widetilde{q} \|_{0,E}^2\Big)^{1/2} \leq C^\prime\, \enorm{q} \, ,
\]
and
\[
    \enorm{\widetilde{q}}\leq C^{\prime\prime}\, \enorm{q}\, . 
\]

Given  that
    \begin{align*}
        &\sum_{K \in \Ch} (\Eop{p_h}+\lambda_h+f,q-\widetilde{q})_{0,K} \\
     &\qquad \leq \sum_{K \in \Ch} \tfrac{h_K}{d_K^{3/2}} \|\Eop{p_h}+\lambda_h+f\|_{0,K}\, \tfrac{d_K^{3/2}}{h_K} \| q-\widetilde{q} \|_{0,K}, \\ 
     &\sum_{E \in \Eh} (d^3 \llbracket \nabla p_h \cdot \boldsymbol{n} \rrbracket,q-\widetilde{q})_{0,E}\\
     &\qquad \leq \sum_{E \in \Eh} \tfrac{h_E^{1/2}}{d_E^{3/2}} \| d^3 \llbracket \nabla p_h \cdot \boldsymbol{n} \rrbracket\|_{0,E} \tfrac{d_E^{3/2}}{h_E^{1/2}} \| q-\widetilde{q} \|_{0,E}, \\ 
     & \sum_{K \in \Ch} \tfrac{h_K^2}{d_K^3}(\Eop{p_h}+\lambda_h+f,\Eop{\widetilde{q}})_{0,K} \\
     & \qquad \leq \sum_{K \in \Ch} \tfrac{h_K}{d_K^{3/2}} \|\Eop{p_h}+\lambda_h+f\|_{0,K}\, \tfrac{h_K}{d_K^{3/2}} \|\Eop{\widetilde{q}}\|_{0,K},
       \end{align*}
   and (cf.~\cite{GSV17})
   \[
   \langle p_h-p_c,\lambda_h-\lambda\rangle \leq \enorm{(p_c-p_h)_+}\edual{\lambda-\lambda_h} + \langle (p_h-p_c)_+,\lambda_h \rangle   \]
   one readily derives the upper bound  for $\enorm{p-p_h}+\edual{\lambda-\lambda_h}$ by combining the above estimates and using the inverse inequality \eqref{inverseineq}.
    
\end{proof}

In the next theorem we show the existence of a lower bound.
\begin{theorem}
  There exists a constant $C_1>0$ such that 
  \begin{align*}
    & \Big( \sum_{K \in \Ch} \tfrac{h_K^2}{d_K^3} \| \Eop{p_h} + \lambda_h + f \|_{0,K}^2\Big)^{1/2}+ 
    \Big(\sum_{E \in \Eh} \tfrac{h_E}{d_E^3} \| d^3 \llbracket \nabla p_h \cdot \boldsymbol{n} \rrbracket\|_{0,E}^2\Big)^{1/2} \\
 &\phantom{=}\qquad \leq  C_1 \left( \enorm{p-p_h} + \edual{\lambda - \lambda_h} +\sqrt{\sum_{K \in \Ch} \tfrac{h_K^2}{d_K^3}\|f-f_h\|_{0,K}^2 } \right)\, .
     \end{align*}
\end{theorem}
\begin{proof} Similarly to~\cite{GSV17}, we first derive local lower bounds. We start  by defining, for any $q_h\in Q_h$ and $\mu_h\in \Lambda_h$ and in each $K$, a function $\gamma_K$ by
\[
\gamma_K = \tfrac{h_K^2}{d_K^{3}} \, b_K \, (\Eop{p_h}+\lambda_h+f_h) 
\]
where $b_K\in P_3(K)\cap H^1_0(K)$ is the usual bubble function, and let
$\gamma_K= 0 $ in $\Omega\setminus K$.  Norm equivalence, equality
\eqref{eq:lagstrong1} and integration by parts, imply that
\begin{align*}
& \tfrac{h_K^2}{d_K^3} \|\Eop{q_h}+\mu_h+f_h\|_{0,K}^2 \\ &
\quad \leq C\,\tfrac{h_K^2}{d_K^{3}}\|\sqrt{b_K}( \Eop{q_h}+\mu_h+f_h)\|_{0,K}^2 = C\, (\Eop{q_h}+\mu_h+f_h, \gamma_K)_K \\
& \quad \leq C\, \big( (d^{3} \nabla (p-q_h), \nabla \gamma_K)_K + \langle\gamma_K, \mu_h-\lambda\rangle + (f_h-f,\gamma_K)_K\big)\\
& \quad \leq C\, \big(\enorm{p-q_h}_K \enorm{\gamma_K}_K + \edualK{\lambda-\mu_h} \enorm{\gamma_K}_K+  \|f-f_h\|_{0,K} \|\gamma_K\|_{0,K} \big) ,
\end{align*}
where we have defined, for any $\mu\in \Lambda$,
\[
\edualK{\mu} = \sup_{q\in H^1_0(K)}  \frac{\langle q,\mu\rangle}{\enorm{q}_K}
\, .
\] 
Here $q\in H^1_0(K)$ has been  extended by zero into $\Omega\setminus K$.
Using inverse estimates and norm equivalence, one easily shows that
\[
\enorm{\gamma_K}_K^2 \leq C\, d_K^3h_K^{-2} \ \|\gamma_K\|_{0,K}^2 \leq C \tfrac{h_K^2}{d_K^3} \|\Eop{q_h}+\mu_h+f_h\|_{0,K}^2\,.
\]
Therefore,
\[
 \|f-f_h\|_{0,K} \|\gamma_K\|_{0,K}  \leq C\, \Big( \tfrac{h_K}{d_K^{3/2}}  \|f-f_h\|_{0,K} \Big)\, \Big( \tfrac{h_K}{d_K^{3/2}} \|\Eop{q_h}+\mu_h+f_h\|_{0,K} \Big).
 \]
 Summing the above estimates over $K\in \mathcal{C}_h$ and choosing $q_h=p_h$ and $\mu_h=\lambda_h$ leads to the global lower bound in terms of the element residuals. We omit the proof of the bound in terms of the edge residuals since it is very similar, albeit a bit more technical, and has been proved in detail for a constant coefficient elliptic variational inequality in~\cite{GSV17}.

 \end{proof}

Aiming at generating sequences of adaptively refined meshes, we define, in view of Theorem~\ref{thm:aposteriori}, an
elementwise error estimator  
\begin{align}
    \mathcal{E}_K^2 &= \frac{h_K^2}{d_K^3} \|\Eop{p_h}+\lambda_h+f\|_{0,K}^2 + \frac12 \frac{h_E}{d_E^3} \| d^3 \llbracket \nabla p_h \cdot \boldsymbol{n} \rrbracket \|_{0,\partial K}^2 \\
                    &\phantom{=}+ \enorm{(p_c-p_h)_+}_{K}^2 + \int_K (p_h-p_c)_+ \lambda_h \dx. \label{eq:estimator}
\end{align}
Next, we choose an initial mesh $\Ch^0$,  a value for the parameter $\beta \in (0,1)$ and a maximum
number of iterations $k_\text{max}$, let $k=1$ and perform the following steps:
\begin{enumerate}[leftmargin=1.5cm]
    \item[Step 1.] Solve the stabilized problem using the mesh $\Ch^{k-1}$.
    \item[Step 2.] Evaluate the error estimator $\mathcal{E}_K$ for each $K \in \Ch^{k-1}$.
    \item[Step 3.] Refine the elements that satisfy $\mathcal{E}_K > \beta \max_{K^\prime \in \Ch^{k-1}} \mathcal{E}_{K^\prime}$ and build $\Ch^k$.
    \item[Step 4.] If $k>k_\text{max}$ stop, otherwise set $k=k+1$ and go to Step 1.
\end{enumerate}
An element is refined in Step 3 by adding a vertex at the
midpoint of each of its three  edges. This larger set of vertices is fed to a mesh
generator~\citep{shewchuk1996triangle} with an option to include additional points
if the shape regularity needs improving.

\section{Implementation}
\label{implementation}

The stabilized method of Problem~\ref{prob:stabfem} is straightforward to implement using a semismooth Newton's method. This leads to an efficient
solution strategy where the reassembly of finite element matrices
at each Newton step can be circumvented~\citep{GSV17, gustafsson2017obstacle}

Alternatively, the Lagrange multiplier, which is discontinuous across the inter-element boundaries, can be eliminated locally  to yield a formulation
 similar to the traditional penalty method.  In fact, letting
$\mathcal{H} \in L^2(\Omega)$, $\mathcal{D} \in L^2(\Omega)$ and
$\Eoph{p_h} \in L^2(\Omega)$ denote the functions satisfying
\begin{equation}
\mathcal{H}|_K = h_K, \quad
\mathcal{D}|_K = d_K, \quad
\Eoph{p_h}|_K = \Eop{p_h|_K}, \quad \forall K \in \Ch ,
\end{equation}
the discrete cavitation region corresponding to $p_h$ becomes
\begin{equation}
  \Omega_h^c(p_h) = \{(x,y) \in \Omega : \tfrac{\mathcal{D}^3}{\alpha \mathcal{H}^2}(p_c-p_h) - f - \Eoph{p_h} > 0 \}.
\end{equation}
Defining  the discrete forms
\begin{align*}
    a_h(p_h,q_h; r_h) &= \int_{\Omega_h^c(r_h)}\!\! \big( \, p_h\,\Eoph{q_h} + \Eoph{p_h} q_h + \tfrac{\mathcal{D}^3}{\alpha \mathcal{H}^2} p_h q_h \big)\dx \\
    &\quad - \alpha \int_{\Omega \setminus\Omega_h^c(r_h)} \mathcal{H}^2 \Eoph{p_h} \Eoph{q_h} \dx, \\
    L_h(q_h; r_h) &= \int_{\Omega_h^c(r_h)}\!\! \big(\,  p_c\,\Eoph{q_h}  + \tfrac{\mathcal{D}^3}{\alpha \mathcal{H}^2} p_c q_h - f q_h \big)\dx \\
    &\quad + \alpha \int_{\Omega \setminus\Omega_h^c(r_h)}\mathcal{H}^2 f \Eoph{q_h} \dx,
\end{align*}
we then arrive at the following formulation of  Problem~\ref{prob:stabfem}.
\begin{problem}[Nitsche's method]
  Find $p_h \in Q_h$ and the region $\Omega_h^c = \Omega_h^c(p_h)$ such that
  \begin{equation}
    \int_{\Omega} d^3 \nabla p_h \cdot \nabla q_h \dx + a_h(p_h,q_h; p_h) = \int_{\Omega} f q_h \dx + L_h(q_h; p_h),
  \end{equation}
  for every $q_h \in Q_h$.
  \label{prob:nitsche}
\end{problem}

\begin{remark}
    \label{rem:iteration}
    The iterative method for solving Problem~\ref{prob:nitsche}  reads as:  given $p_h^{k-1} \in Q_h$, find
     $p_h^k \in Q_h, k=1,2,\ldots$ such that
    \begin{equation}
        \int_{\Omega} d^3 \nabla p_h^k \cdot \nabla q_h \dx + a_h(p_h^k,q_h; p_h^{k-1}) = \int_{\Omega} f q_h \dx + L_h(q_h; p_h^{k-1}),
    \end{equation}
    for every $q_h \in Q_h$. 
    One starts with an initial guess $p_h^0$ which solves the
    unconstrained Reynolds equation and the iteration is terminated when
    $\enorm{p_h^k-p_h^{k-1}}$ is small enough. The a posteriori error estimator
    defined in \eqref{eq:estimator} applies with  $\lambda_h =
    (\frac{\mathcal{D}^3}{\alpha \mathcal{H}^2}(p_c - p_h) - f - \Eoph{p_h})_+$.
\end{remark}

\begin{remark}
A similar method was recently proposed by Burman et al. in \cite{burman2017} where the authors recast the classical obstacle problem as an equality for the primal variable. Their method is based on  an augmented Lagrangian approach through an elimination of the Lagrange multiplier in the spirit of Chouly and Hild \cite{choulyhild2013}. The extra terms can interpreted as nonlinear consistent penalty terms. 
\end{remark}

\section{Numerical results}

The work of Raimondi--Boyd~\citep{raimondi1958solution} contains a
well-documented and classical application of the Reynolds cavitation model. The
authors solve a finite-width journal bearing problem for various configurations
using the finite difference method, essentially extending the solution
technique of Christopherson~\citep{christopherson1941new} to two-dimensional
problems. We will consider the same boundary-value problem to highlight the
differences between the classical and the proposed numerical procedures. We
 solve the problem also by the finite element penalty method for which a
posteriori error estimators are available and which has often been considered
for approximating cavitation problems. 

We consider a rectangular domain $\Omega = [0,\frac{2\pi}{3}] \times
[0,1]$.  The boundary conditions are $p = 0$ on the whole boundary
$\partial \Omega$ and the cavitation pressure is $p_c=0$. The geometry
of the bearing with length-to-radius ratio $L/R = 2$ gives the
following operator and loading:
\begin{equation}
    \label{eq:opnload}
    \Eop{p} = \frac{\partial}{\partial \theta}\left( d^3 \frac{\partial p}{\partial \theta}\right) - \frac{1}{4} \frac{\partial}{\partial y}\left(d^3 \frac{\partial p}{\partial y}\right), \quad f = -6 \frac{\partial d}{\partial \theta},
\end{equation}
where $d(\theta) = 1 + \varepsilon \cos (\theta - \varphi)$,
$\varepsilon = 0.9$ and $\varphi = 0.5483\overline{88}$.
Note that due to the chosen eccentricity, the values of the field
$d^3$ vary approximately three orders of magnitude inside the
domain.

We solve the described problem using the Nitsche's and penalty methods
with linear and quadratic elements. Solutions using higher order
elements can be computed by our implementation but their use is
counterproductive since the solution has limited regularity.  Recall
that the penalty method for this problem is essentially the method of
Problem~\ref{prob:nitsche} with the definitions
\begin{align*}
  \Omega_h^c(r_h) &= \{(x,y) \in \Omega : p_h < 0 \}, \\
    a_h(p_h,q_h; r_h) &= \int_{\Omega_h^c(r_h)}\!\! \tfrac{1}{\epsilon \mathcal{H}^{k+1}} p_h q_h \dx, \\
  L_h(q_h; r_h) &= 0,
\end{align*}
where $k$ is the polynomial degree of the finite element basis and $\epsilon >
0$ is the penalty parameter.  The error estimator for the penalty method is
\begin{equation}
    \mathcal{E}_K^2 = \frac{h_K^2}{d_K^3} \|\Eop{p_h}+\tfrac{1}{\epsilon \mathcal{H}^{k+1}}(-p_h)_++f\|_{0,K}^2 + \frac12 \frac{h_E}{d_E^3} \| d^3 \llbracket \nabla p_h \cdot \boldsymbol{n} \rrbracket \|_{0,\partial K}^2.
\end{equation}
In each case the mesh is repeatedly refined as described in
Section~\ref{sec:adaptive}. The refinement parameter is chosen as $\beta =
0.5$. The stabilization parameter for the Nitsche's method is taken to be $\alpha
= 10^{-2}$ which is suitably small for both polynomial degrees. The penalty
parameter is chosen quite large,  $\epsilon = 10$, since  smaller values seem to cause
problems for the convergence of the iteration process described in
Remark~\ref{rem:iteration}---possibly due to the increasing condition number
of the system matrix. The
iteration counts with respect to the stabilization/penalty parameter
are comparable between the two methods. In the
penalty method, the number of iterations can be reduced by increasing the value of the
penalty parameter which, on the other hand, decreases the accuracy.
Nitsche's method converges fastest if
the stabilization parameter is chosen to be close to its upper bound
$C_I$.

The discrete solution, computed with a fine mesh, is visualized in
Figure~\ref{fig:lipsum0}. 
The pressure isolines depicted in Raimondi--Boyd~\cite{raimondi1958solution} are very similar in shape to ours,
and their maximum pressure value 32.8  is close to the one obtained
by our implementation, $\max_{\boldsymbol{x} \in \Omega} p_h(\boldsymbol{x})
\approx 32.750$. This suggests that the operator and the loading given in
\eqref{eq:opnload} have been interpreted correctly for the problem at hand.

Starting with the initial mesh given in Figure~\ref{fig:lipsum1}~(a), we solve,
mark and refine four times using the four different methods. The resulting
meshes are shown in Figure~\ref{fig:lipsum1}~(b--e).  Observe that, with linear elements,  the 
meshes of the stabilized and penalty methods  are very similar except for
 some  refinement in the penalty solution in the upper and
lower right-hand corners of the cavitated region where, at least intuitively, one would not expect
 a constant solution needing additional triangles.
The problem of unnecessary refinement is more visible when using quadratic
elements.  The Nitsche's method leads to a satisfactory mesh with an emphasis on
the cavitation boundary where the solution is less regular. The penalty method
refines primarily and incorrectly in the area where the solution is known to be
constant.

To demonstrate this effect in a more quantitative manner, we start
with a coarser mesh and plot the square root of the sum of the error estimators
as a function of the number of degrees of freedom, see Figure~\ref{fig:error}.
Note that the sum of the estimators is known to be an upper bound for the true
error as indicated by the a posteriori estimates. Again  with
linear elements both methods perform quite similarly whereas in the quadratic case the
Nitsche's method has a clear edge. The explanation lies in the non-consistency
of the penalty approach where the consistency error ultimately limits the rate
of convergence, cf.~\cite{gustafsson2017obstacle}.

\section{Conclusions}

A stabilized method was presented for the adaptive finite element solution of an inequality-constrained Reynolds equation, modeling cavitation in hydrodynamic lubrication, as an alternative to the classical penalty approach.  Optimal a priori estimates and new a posteriori error estimators, suitably modified for the governing variable-coefficient obstacle problem, were discussed and the latter were used in numerical experiments. 

The positivity constraint on the pressure field was  imposed using a Lagrange multiplier which, for being discontinuous, can be eliminated elementwise at the discrete level. Hence, the stabilized method could be implemented as a Nitsche-type method which, contrary to the penalty approach, leads to an optimally conditioned, symmetric and positive-definite stiffness matrix. The Nitsche's method was shown to avoid over-refinement and be well suited for resolving the free boundary related to the initially unknown cavitation region and thus win over the classical penalty method, especially if a quadratic basis is employed. 
  
\medskip

\noindent{\bf Acknowledgements.}
Funding from Tekes -- the Finnish Funding Agency for Innovation (Decision number 3305/31/2015) and the Finnish Cultural
    Foundation is gratefully acknowledged, as well as the financial support from FCT/Portugal through UID/MAT/04459/2013 and from the UT Austin-Portugal CoLab Program.

\begin{figure}
    \centering
    \includegraphics[width=\textwidth]{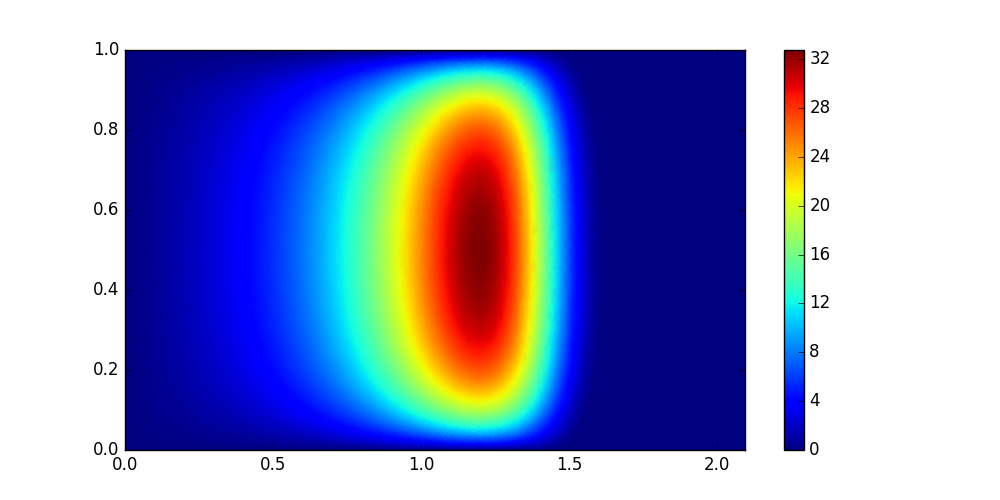}\\
    \includegraphics[width=0.64\textwidth]{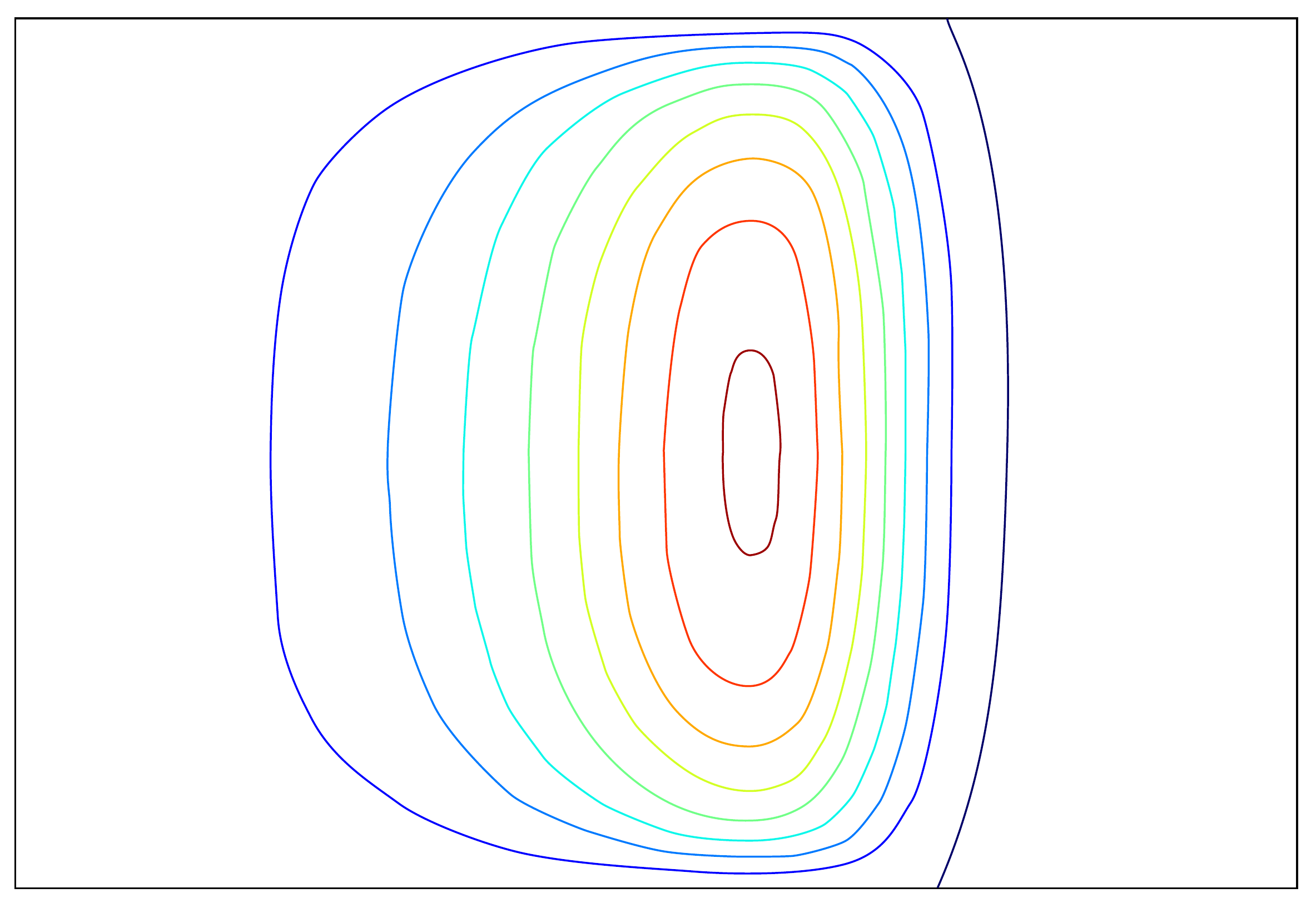}\hspace{1.5cm}
    \caption{The discrete solution after seven refinements using Nitsche's
        method and linear elements is given in the upper panel. The lower panel
        depicts the pressure isolines at values $0, 4, 8, \dots, 32$ as given in
        Raimondi--Boyd~\cite{raimondi1958solution} where the authors report a
        maximum pressure value of 32.8. Our adaptive method gives the maximum
        value
of 32.750.}
    \label{fig:lipsum0}
\end{figure}

\begin{figure}
    \centering
    \subcaptionbox{Initial mesh}{\includegraphics[width=0.48\textwidth]{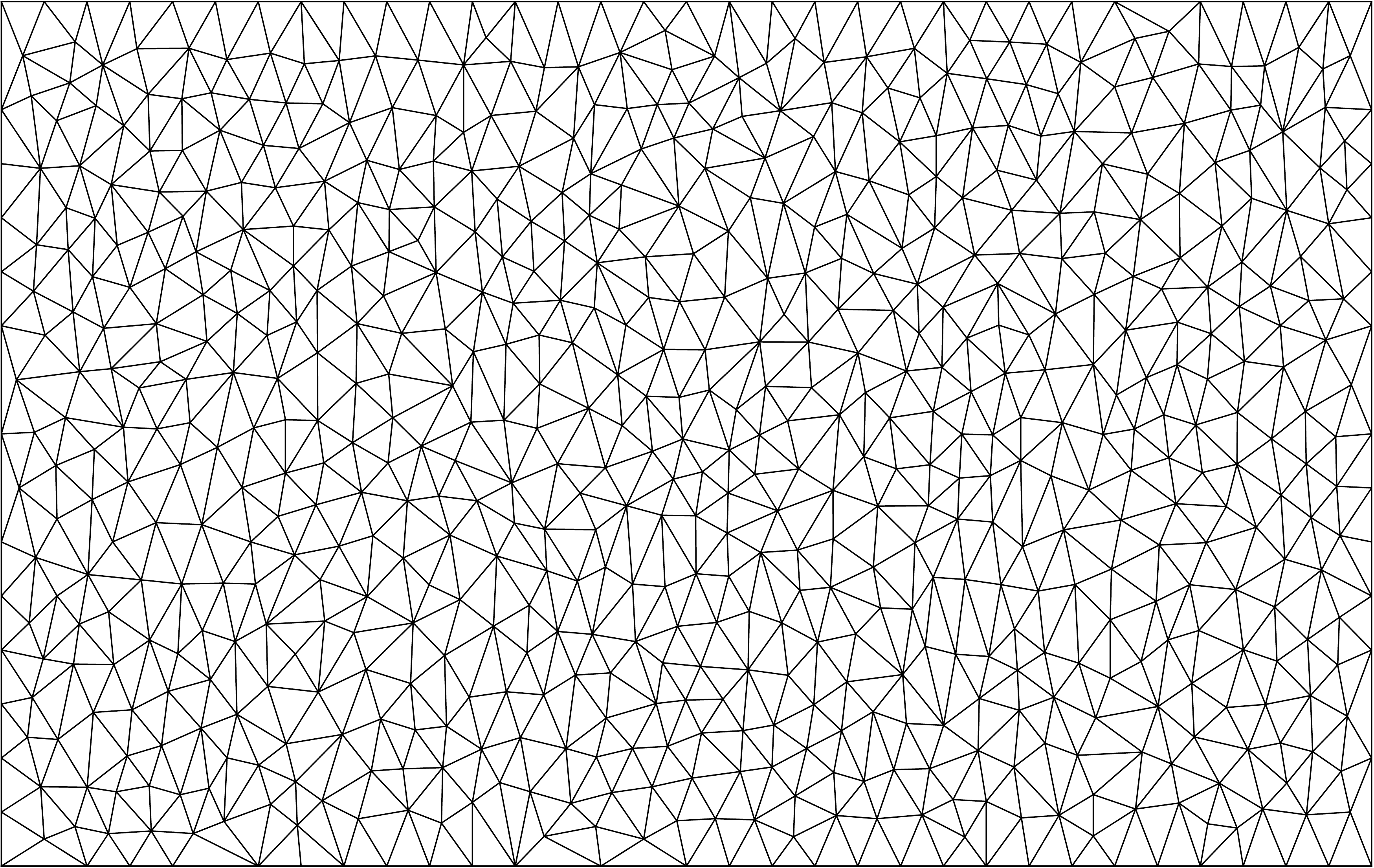}}\\[0.3cm]
    \subcaptionbox{Nitsche, linear elements}{\includegraphics[width=0.48\textwidth]{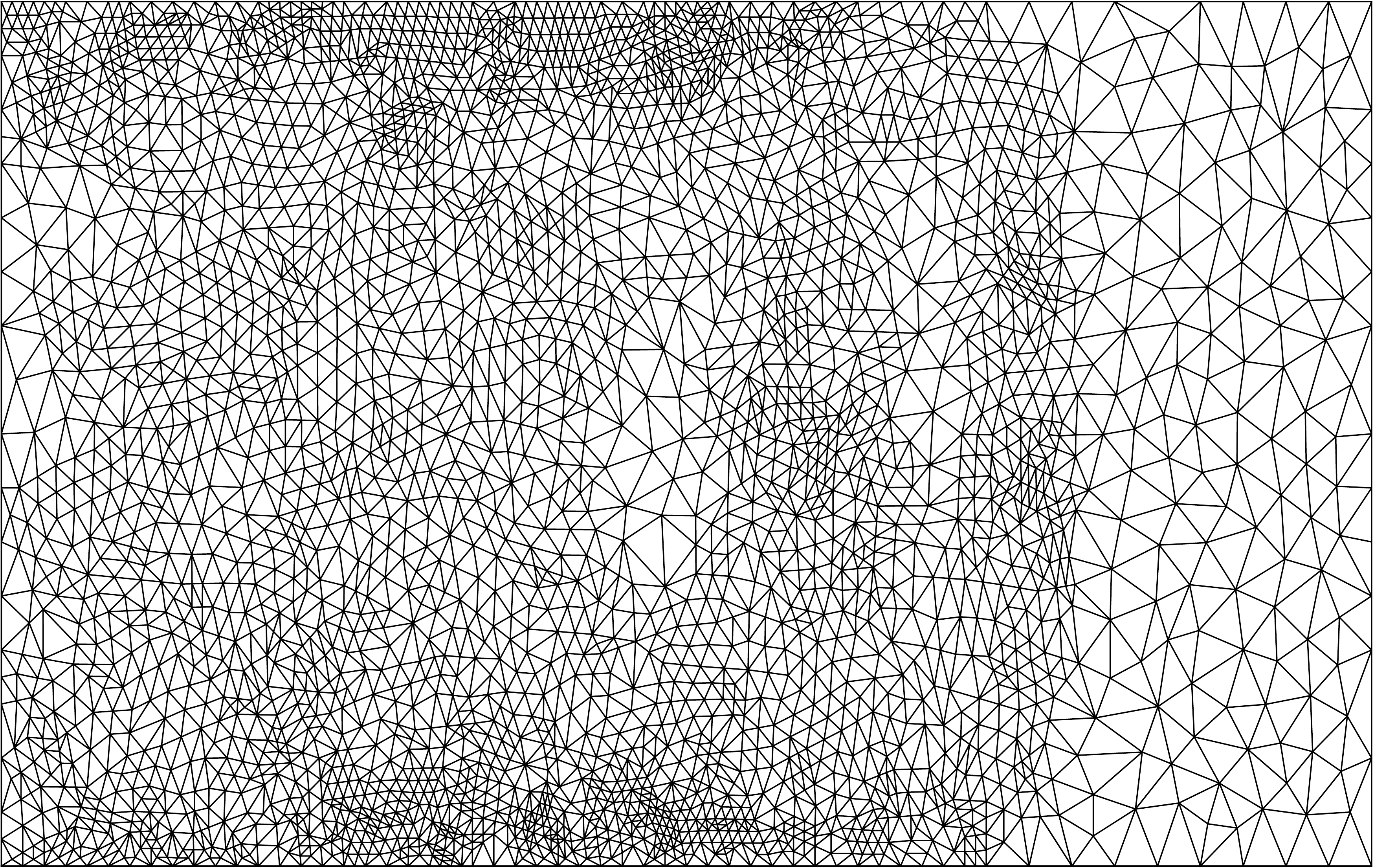}}\hfill
    \subcaptionbox{Nitsche, quadratic elements}{\includegraphics[width=0.48\textwidth]{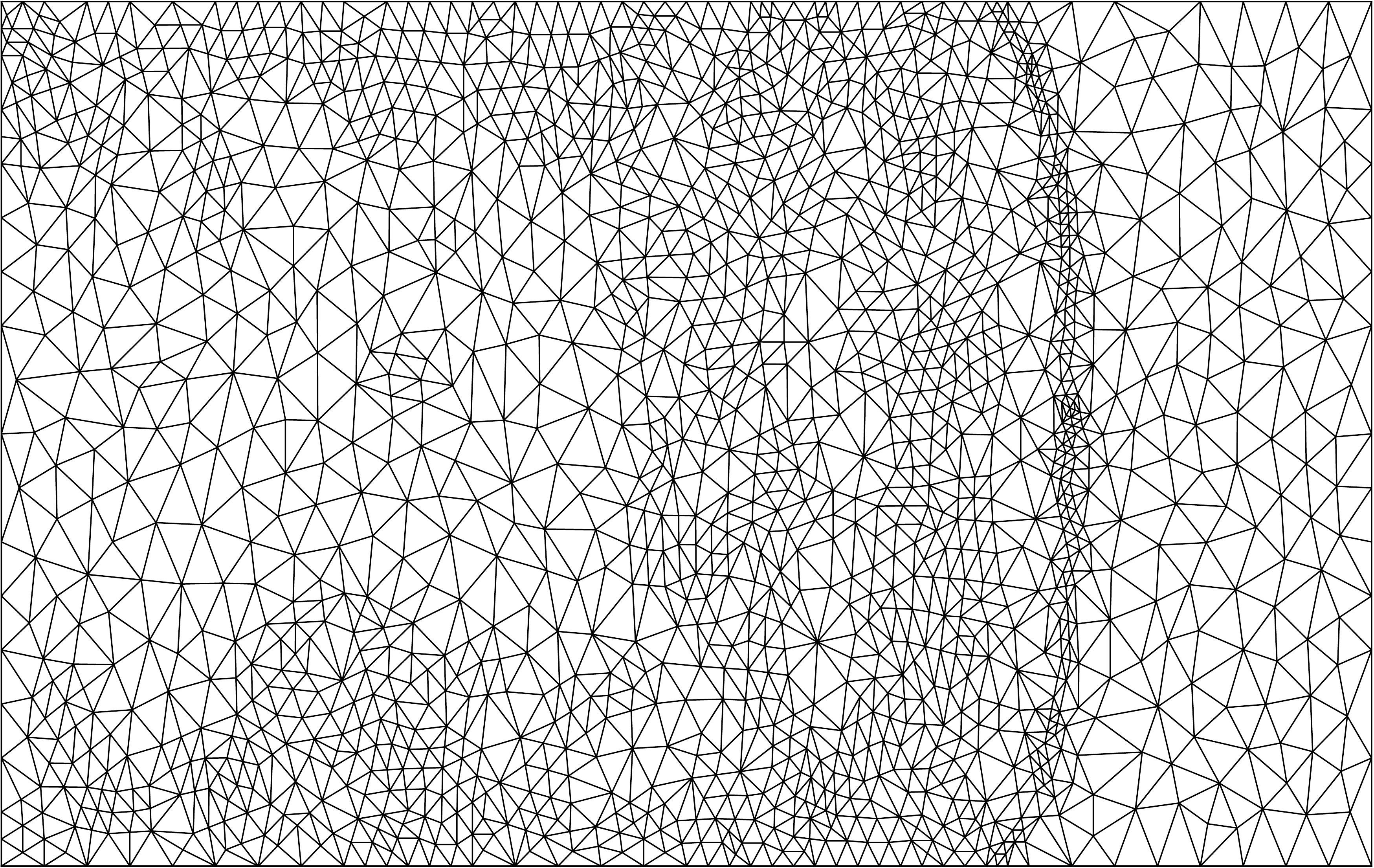}}\\[0.3cm]
    \subcaptionbox{Penalty, linear elements}{\includegraphics[width=0.48\textwidth]{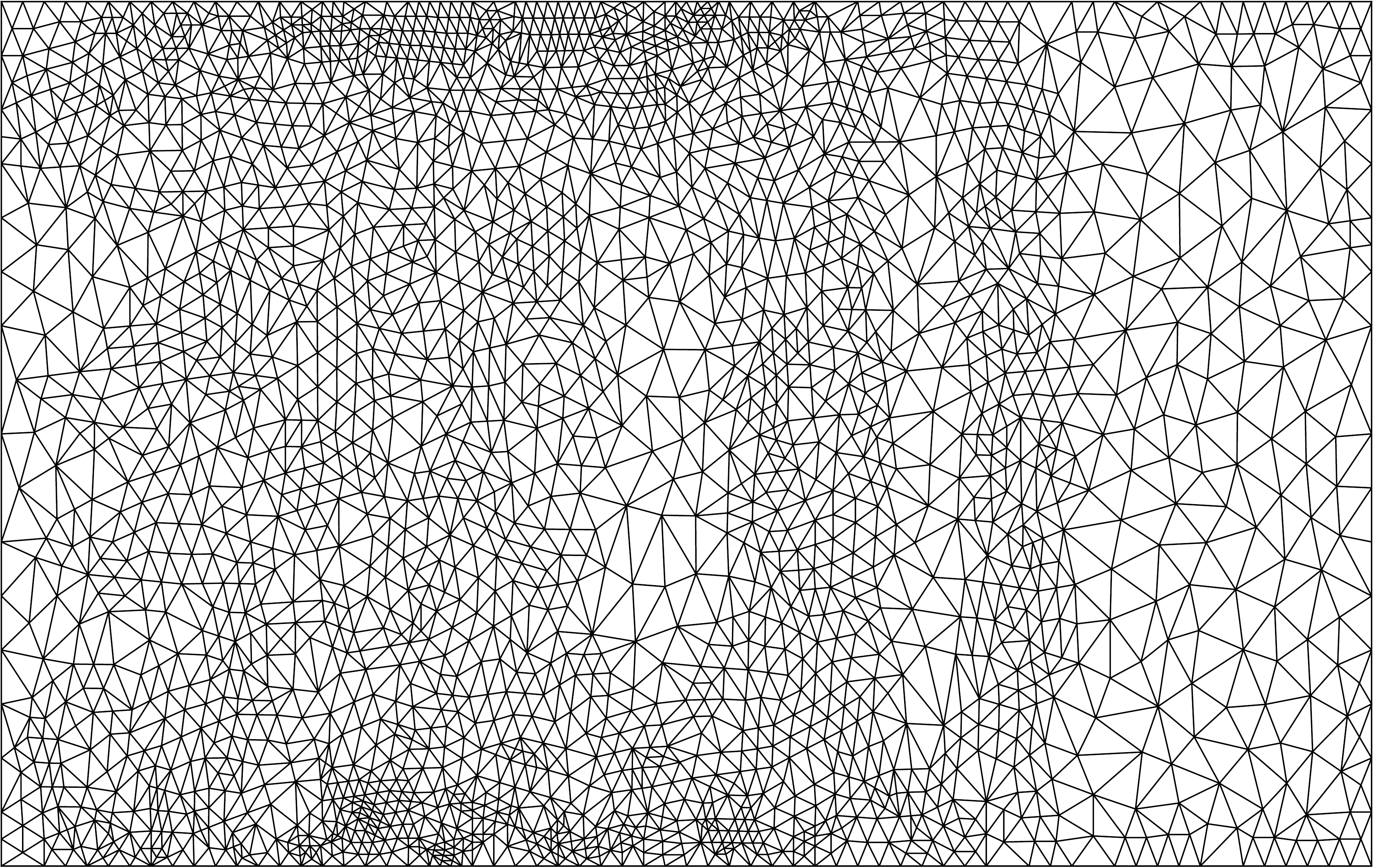}}\hfill
    \subcaptionbox{Penalty, quadratic elements}{\includegraphics[width=0.48\textwidth]{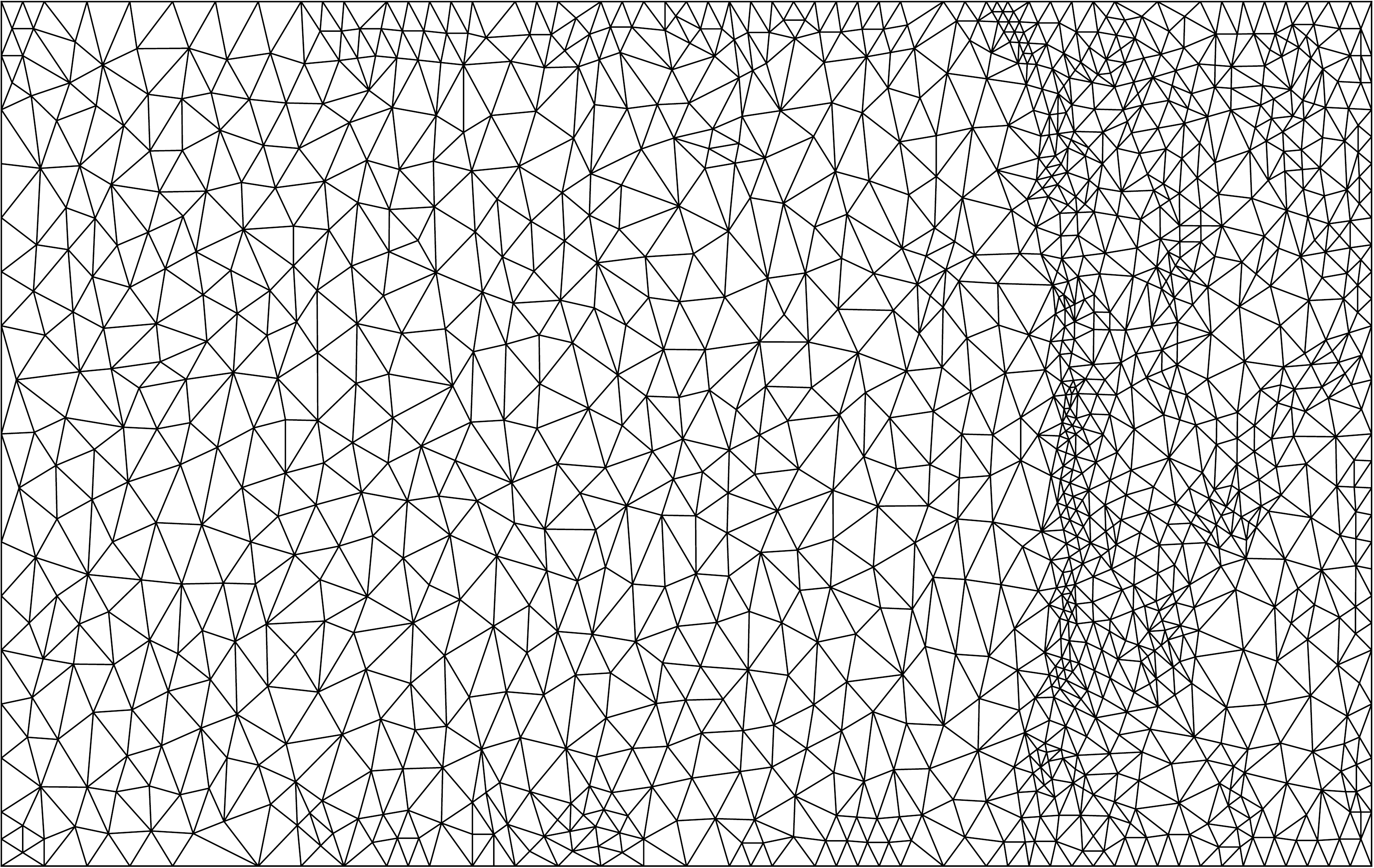}}
    \caption{Initial and final meshes of the different methods after four mesh
    refinements.}
    \label{fig:lipsum1}
\end{figure}

\pgfplotstableread{
    ndofs eta
    51 8.94040244036
    129 5.56426178412
    249 3.89443152365
    422 2.96459822407
    796 2.14421652957
    1442 1.58848630981
    1901 1.37390993234
    3372 1.04862845095
}\nitschelinear

\pgfplotstableread{
    ndofs eta
    51 9.7979073614
    138 5.94510800317
    298 3.87671199846
    483 3.04270597908
    939 2.17230064832
    1101 1.96681845357
    2331 1.35959753088
}\penaltylinear

\pgfplotstableread{
    ndofs eta
    181 3.58236796758
    250 2.29024071127
    463 1.17115475304
    666 0.802267828735
    1190 0.44398170033
    1533 0.311389566186
    2370 0.199833946208
    3819 0.122491529401
}\nitschequadratic

\pgfplotstableread{
    ndofs eta
    181 4.29639181502
    260 3.15322349584
    517 1.70139802359
    887 1.0392762071
    1473 0.648802573383
    2143 0.462643186662
    2587 0.390153530888
    3846 0.300306171689
}\penaltyquadratic

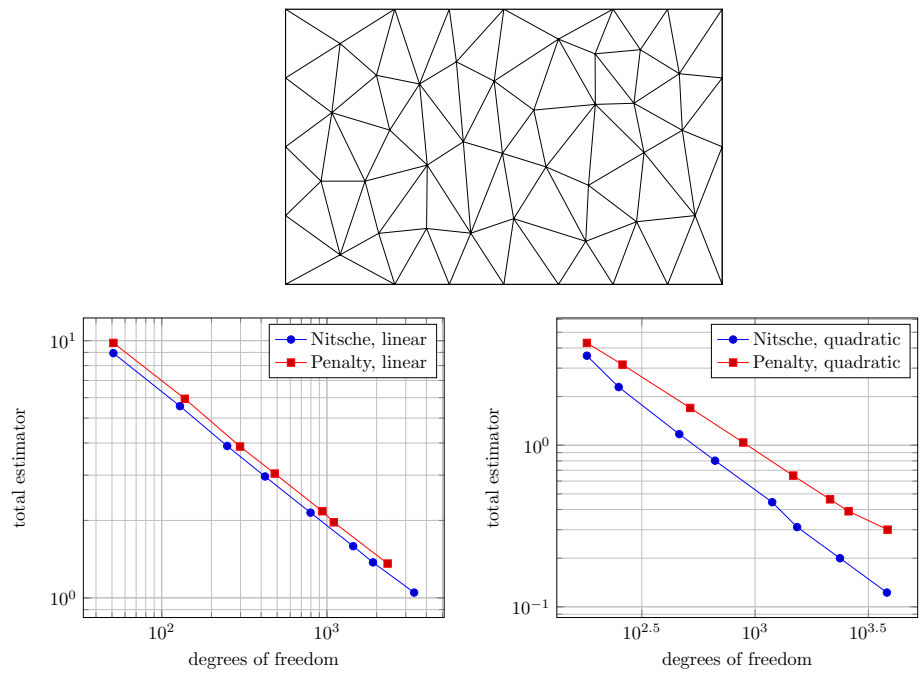
\begin{figure}
    \centering
    \hspace{1cm}
    \includegraphics[width=0.48\textwidth]{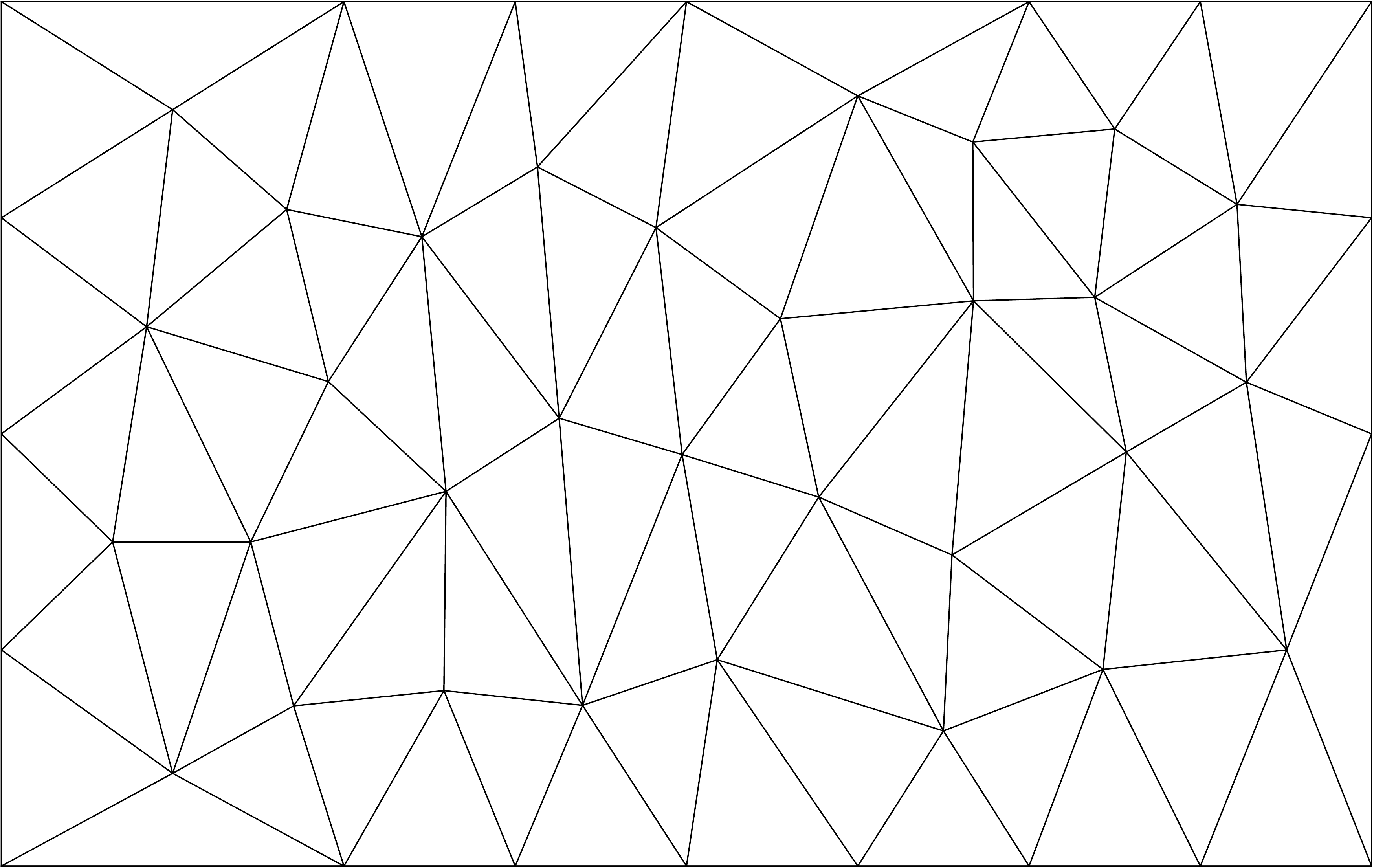}\\[0.4cm]
    \begin{tikzpicture}[scale=0.7]
        \begin{axis}[
                xmode = log,
                ymode = log,
                xlabel = {degrees of freedom},
                ylabel = {total estimator},
                grid = both
            ]
            \addplot table[x=ndofs,y=eta] {\nitschelinear};
            \addplot table[x=ndofs,y=eta] {\penaltylinear};
            \addlegendentry{Nitsche, linear}
            \addlegendentry{Penalty, linear}
        \end{axis}
    \end{tikzpicture}\hfill
    \begin{tikzpicture}[scale=0.7]
        \begin{axis}[
                xmode = log,
                ymode = log,
                xlabel = {degrees of freedom},
                ylabel = {total estimator},
                grid = both
            ]
            \addplot table[x=ndofs,y=eta] {\nitschequadratic};
            \addplot table[x=ndofs,y=eta] {\penaltyquadratic};
            \addlegendentry{Nitsche, quadratic}
            \addlegendentry{Penalty, quadratic}
        \end{axis}
    \end{tikzpicture}
    \caption{The initial mesh (upper panel) and the resulting global error estimators
    for the two methods (lower panels), using linear and quadratic elements.  The global estimator
    is defined as the square root of the sum of the elementwise estimators.}
    \label{fig:error}
\end{figure}


\bibliographystyle{elsarticle-num}
\bibliography{reynolds}

\end{document}